\documentclass[12pt]{article}

\usepackage{amsfonts}
\usepackage{amssymb}
\usepackage{amsthm}
\usepackage{amsmath}
\usepackage{mathrsfs}
\usepackage{xcolor}
\usepackage[all]{xy}
\usepackage[utf8]{inputenc}
\def\Vect{\hbox{-}\mathsf{Vect}}
\def\Mod{\hbox{-}\mathsf{Mod}}
\def\un{\mathbf{1}}
\def\mpn{\medskip\par\noindent}
\def\dom{\backslash}
\newcommand{\gMod}[1]{#1{\hbox{-}\mathsf{Mod}}}
\def\zero{\{0\}}
\def\sou{\underline}

\usepackage[left=1.20in,right=1.20in]{geometry}
\definecolor{ao(english)}{rgb}{0.0, 0.5, 0.0}
\definecolor{brickred}{rgb}{0.8, 0.25, 0.33}
\definecolor{burntorange}{rgb}{0.8, 0.33, 0.0}
\definecolor{beaver}{rgb}{0.62, 0.51, 0.44}
\definecolor{brown(traditional)}{rgb}{0.59, 0.29, 0.0}
\definecolor{ao(english)}{rgb}{0.0, 0.5, 0.0}
\definecolor{verde}{rgb}{0.12, 0.8, 0.17}
\def\rouge{\color{red}}

\def\sou{\underline}

\newcommand{\rac}{\mathbb Q}

\newcommand{\cp}{\mathcal P}

\newcommand{\Ind}{\mathrm{Ind}}
\newcommand{\Res}{\mathrm{Res}}
\newcommand{\Def}{\mathrm{Def}}
\newcommand{\Inf}{\mathrm{Inf}}
\newcommand{\Iso}{\mathrm{Iso}}
\newcommand{\Hom}{\mathrm{Hom}}

\newcommand{\End}{\mathrm{End}}
\newcommand{\Id}{\mathrm{Id}}

\newcommand{\Rad}{\mathrm{Rad}}

\theoremstyle{plain}
\newtheorem{teo}{Theorem}[section]
\newtheorem*{teo-non}{Theorem}
\newtheorem{prop}[teo]{Proposition}
\newtheorem{coro}[teo]{Corollary}
\newtheorem{lema}[teo]{Lemma}

\theoremstyle{definition}
\newtheorem{defi}[teo]{Definition}
\newtheorem{nota}[teo]{Notation}

\theoremstyle{remark}
\newtheorem{ejem}[teo]{Example}

\newtheorem{rem}[teo]{Remark}

\def\CD{\mathcal{D}}
\def\CP{\mathcal{P}}

\def\C{\mathbb{C}}
\def\F{\mathbb{F}}
\def\Q{\mathbb{Q}}
\def\Z{\mathbb{Z}}
\def\mpn{\medskip\par\noindent}

\def\op{^{\mathrm{op}}}
\def\endpf{~\leaders\hbox to 1em{\hss\  \hss}\hfill~\raisebox{.5ex}{\framebox[1ex]{}}\smallskip\par}

\reversemarginpar
\author{Serge Bouc and Nadia Romero\\
}
\title{Green fields}
\date{}

\begin{document}

\maketitle
\begin{abstract}
We introduce \textit{Green fields}, as commutative Green biset functors with no non-trivial ideals. We state some of their properties and give examples of known Green biset functors which are Green fields. Among the properties, we prove some criterions ensuring that a Green field is semisimple. Finally, we describe a type of Green field for which its category of modules is equivalent to a category of vector spaces over a field. \vspace{1ex}
\end{abstract}
\section{Introduction}
A Green biset functor can be defined as a ring object in the tensor category  of biset functors (\cite{biset} Section 8.5), so it is natural to extend the usual definitions for rings to the realm of Green biset functors. In particular, we have a notion of left (or right, or two-sided) ideal of a Green biset functor. In~\cite{centros}, we consider two analogues of the center of a ring, and we consider {\em commutative} Green biset functors. For such functors, the notions of left, right, and two-sided ideal coincide. In the present paper, we study commutative Green biset functors with no non-trivial ideals, and we call these functors {\em Green fields} (this terminology first appeared in \cite{relative-B-groups} Remark 7.21). This definition allows us to continue the search for ways  of decomposing the modules over a Green biset functor into simpler factors, following the spirit of what is done in Section~5 of \cite{centros}, but in a different fashion. In the case of a Green field $A$, the question is, naturally, if $A$-modules behave as vector spaces over a field. As we will see in Section~3, the question of whether a module over a Green field is always semisimple or not is a complex one and we must, for the moment, leave it open. Nonetheless, we prove some criterions given conditions for this to happen. On the other hand, in Section~4, we introduce {\em strict Green fields}, as Green fields fulfilling one extra condition, and we prove that the category of modules over a strict Green field $A$ is equivalent to a category of vector spaces over the field $A(\un)$.

The paper is organized as follows: in Section 2, we recall the basic definitions on Green biset functors and modules over them. In Section 3, we introduce Green fields, state their first properties and enumerate some examples. We also introduce bilinear forms which allow for a characterization of Green fields among all commutative Green biset functors. Next, we give criterions ensuring that a Green field is semisimple. Finally, in Section~4, we define strict Green fields, as those Green fields satisfying a list of six equivalent conditions. We conclude with some examples of strict and non strict Green fields.


\section{Preliminaries}

We assume that the reader is familiar with the definitions of biset category, biset functor, Green biset functor and module over a Green biset functor. Examples and some properties of Green biset functors and their modules can be found in sections 1.1  and 1.2 of \cite{centros}. As in that article, the Green biset functors we consider here are defined on a class $\mathcal{D}$ of finite groups, closed under subquotients and direct products. The biset category is defined over a commutative unital ring $k$. Unless otherwise stated, when we write \textit{Green biset functor} we will mean   \textit{Green $\CD$-biset functor over $k$}, for some $\mathcal{D}$ and~$k$. Similarly, when we speak of the evaluation $A(G)$ of a Green biset functor at a finite group $G$, it will be understood that $G$ is in the class $\CD$. The Burnside functor will be denoted by $kB$, as usual, and if $\mathbb{F}$ is a field of characteristic $0$, the functor of linear $\mathbb{F}$-representations will be denoted by $kR_\mathbb{F}$.

The trivial group will be denoted by $\un$.  For a finite group $G$, we denote by $\Delta(G)$ the diagonal subgroup of $G\times G$, that is $\Delta(G)=\{(g,g)\mid g\in G\}$.

We recall the definition of the category $\mathcal{P}_A$, associated to the Green biset functor~$A$.

\begin{defi} \label{PA}Let $A$ be a Green  biset functor with identity element $\varepsilon\in A(\un)$. The category $\cp_A$ is defined in the following way:
\begin{itemize}
 \item The objects of $\cp_A$ are all finite groups {in $\CD$}.
 \item If $G$ and $H$ are groups {in $\CD$}, then ${\Hom}_{\cp_A}(H,\, G)=A(G\times H)$.
 \item Let $H,\, G$ and $K$ be groups {in $\CD$}. The composition of $\beta\in A(H\times G)$ and $\alpha\in A(G\times K)$ in $\cp_A$ is the following:
\begin{displaymath}
\beta \circ \alpha = A\left(\Def^{\,H\times\Delta(G)\times K}_{H\times K}\circ\Res^{H\times G\times G\times K}_{H\times\Delta(G)\times K}\right)(\beta\times\alpha).
\end{displaymath}
\item For a group $G$ in $\CD$, the identity morphism $\varepsilon_G$ of $G$ in $\cp_A$ is $A(\Ind_{\Delta(G)}^{G\times G}\circ\nolinebreak \Inf_\un^{\,\Delta (G)})(\varepsilon)$.
\end{itemize}
\end{defi}

Recall that if $A$ is a Green biset functor, then for any group $G$ in $\CD$, the evaluation $A(G)$ is a ring for the ``dot'' product defined by
$$\forall \alpha,\beta\in A(G),\;\alpha\cdot\beta=A(\Iso_{\Delta(G)}^G)A(\Res_{\Delta(G)}^{G\times G})(\alpha\times\beta).$$
\begin{rem}\label{trivial group} In the special case $G=\un$, after identifying $G\times G$ with $G$, we then have three products $A(G)\times A(G)\to A(G)$, namely $(\alpha,\beta)\mapsto \alpha\times\beta$, $(\alpha,\beta)\mapsto \alpha\circ\beta$, and $(\alpha,\beta)\mapsto \alpha\cdot\beta$. One verifies easily that these products coincide.
\end{rem}

For a morphism $\alpha\in A(H\times G)$ from $G$ to $H$ in the category $\CP_A$, we denote by $\alpha\op$ the element of $A(G\times H)$ - i.e. the morphism from $H$ to $G$ - defined by
$$\alpha\op=A(\Iso_{H\times G}^{G\times H})(\alpha),$$
where $\Iso_{H\times G}^{G\times H}$ is the group isomorphism $H\times G\to G\times H$ swapping the components. The assignment sending a finite group to itself, and a morphism $\alpha$ to $\alpha\op$ is not a functor from $\CP_A$ to the opposite category in general, but it is if $A$ is commutative.
A \textit{commutative} Green biset functor is a Green biset functor $A$ which is equal to its commutant (Definition 21 in \cite{centros}), which means that the product $\times$ of $A$ is commutative in the following sense:
$$\forall G,H\in \CD,\forall \alpha\in A(G),\forall \beta\in A(H),\;\alpha\times\beta=A(\Iso^{G\times H}_{H\times G})(\beta\times\alpha).$$
This is also equivalent to saying that the ring $\big(A(G),\cdot\big)$ is commutative for every group~$G$ in~$\CD$.
\begin{lema} \label{opposite}Let $A$ be a commutative Green biset functor, and $H,K,L\in \CD$. Then for any $\beta\in A(L\times H)$ and $\alpha\in A(H\times K)$
$$(\beta\circ\alpha)\op=\alpha\op\circ\beta\op.$$
The assigment sending $K\in\CD$ to itself and $\alpha\in A(H\times K)$ to $\alpha\op\in A(K\times H)$ is an equivalence of categories from $\CP_A$ to the opposite category.
\end{lema}
\begin{proof}
We compute $\alpha\op\circ\beta\op$:
\begin{align*}
\alpha\op\circ\beta\op&=A(\Def^{K\Delta(H)L}_{KL}\Res_{K\Delta(H)L}^{KHHL})(\alpha\op\times\beta\op)\\
&=A(\Def^{K\Delta(H)L}_{KL}\Res_{K\Delta(H)L}^{KHHL})\big(A(\Iso_{HK}^{KH})(\alpha)\times A(\Iso_{LH}^{HL})(\beta)\big)\\
&=A(\Def^{K\Delta(H)L}_{KL}\Res_{K\Delta(H)L}^{KHHL}\Iso_{HKLH}^{KHHL})(\alpha\times\beta),\\
\end{align*}
where $\Iso_{HKLH}^{KHHL}$ sends $(h,k,l,h')\in HKLH$ to $(k,h,h',l)\in KHHL$. Moreover since $A$ is commutative, we have
$$\alpha\times\beta=A(\Iso_{LHHK}^{HKLH})(\beta\times\alpha),$$
where $\Iso_{LHHK}^{HKLH}$ maps $(l,h,h',k)$ to $(h',k,l,h)$. It follows that
$$
\alpha\op\circ\beta\op=A(\Def^{K\Delta(H)L}_{KL}\Res_{K\Delta(H)L}^{KHHL}\Iso_{LHHK}^{KHHL})(\beta\times\alpha),$$
where $\Iso_{LHHK}^{KHHL}$ maps $(l,h,h',k)$ to $(k,h',h,l)$. Hence
\begin{align*}
\alpha\op\circ\beta\op&=A(\Def^{K\Delta(H)L}_{KL}\Iso_{L\Delta(H)K}^{K\Delta(H)L}\Res_{L\Delta(H)K}^{LHHK})(\beta\times\alpha)\\
&=A(\Iso_{LK}^{KL}\Def_{LK}^{L\Delta(H)K}\Res_{L\Delta(H)K}^{LHHK})(\beta\times\alpha)\\
&=A(\Iso_{LK}^{KL})(\beta\circ\alpha)\\
&=(\beta\circ\alpha)\op,
\end{align*}
as was to be shown. The last assertion follows from the fact that, with a slight abuse of language, the map $\alpha\mapsto \alpha\op$ is obviously involutive, and sends the identity morphism of any object to itself.
\end{proof}

For a Green biset functor $A$, an $A$-module is defined as a biset functor $M$, together with  bilinear products $A(G)\times M(H)\rightarrow M(G\times H)$ that satisfy natural conditions of associativity, identity element and functoriality.
 It is well known (see \cite{biset}, Chapter 8 or \cite{lachica}) 
 that this definition is equivalent to defining an $A$-module  as  a $k$-linear functor from the category $\cp_A$ to $\gMod{k}$. We denote the category of $A$-modules by  $\gMod{A}$. Important objects in $\gMod{A}$ are the \textit{shifted functors}  (also called {\em shifted modules}) $M_L$. If $L$ is a fixed finite group, the functor $M_L$ is defined as $M(G\times L)$ in a group $G\in \mathcal{D}$ and as $M(\alpha \times L)$ in an arrow $\alpha\in A(H\times G)$, for more details see Definition 10 in \cite{centros}. Of course, $A_L$ is an example of this construction but in this case we can say a little bit more. If $L$ is a group in $\mathcal{D}$ and $\mathcal{D}'$ be a subclass of $\mathcal{D}$  possibly not containing $L$, the shifted functor $A_L$ 
is a Green $\mathcal{D}'$-biset functor. 

When dealing with the simple objects of $\gMod{A}$, the following notions are crucial.

\begin{defi}
Let $A$ be a Green biset functor. For a group $H\in \mathcal{D}$, the \textit{essential algebra}, $\widehat{A}(H)$, of $A$ on $H$, is the quotient of $A(H\times H)$ over the ideal generated by elements of the form $a\circ b$, where $a$ is in $A(H\times K)$, $b$ is in $A(K\times H)$ and $K$ runs over the groups in $\mathcal{D}$ of order smaller than $|H|$.
\end{defi}

The following functors can  be defined in more general settings, we recall the definitions in the context of $A$-modules. 

\begin{defi}
Let $A$ be a Green biset functor, $H$ a group in $\mathcal{D}$ and $V$ an $A(H\times H)$-module, the $A$-module $L_{H,\, V}$ is defined in $G\in \mathcal{D}$ as
\begin{displaymath}
L_{H,\, V}(G)=A(G\times H)\otimes_{A(H\times H)}V,
\end{displaymath}
and in an obvious way in arrows.
\end{defi}

If $V$ is a simple $A(H\times H)$-module, then  $L_{H,\,V}$ has a unique maximal subfunctor $J_{H,\,V}$, which in each evaluation is equal to 
\begin{displaymath}
J_{H,\, V}(G)=\left\{\sum_{i=1}^na_i\otimes v_i\in L_{H,\, V}(G)\mid\sum_{i=1}^n(b\circ a_i) v_i=0\ \forall b\in A(H\times G) \right\}.
\end{displaymath}

The quotient $S_{H,\, V}= L_{H,\,V}/J_{H,\,V}$ is a simple $A$-module such that $S_{H,\, V}(H)=V$. On the other hand, if $S$ is a simple $A$-module and $H\in \mathcal{D}$ is such that $S(H)\neq \zero$, then  taking $V=S(H)$ gives $S\cong S_{H,\,V}$.

Finally, we recall the following easy lemma, which will be used in Section~\ref{estrictos}.
\begin{lema}
\label{section}
 Let $\mathcal{E}$ and $\mathcal{F}$ be categories, let $F_1$ and $F_2$ be functors $\mathcal{E}\to\mathcal{F}$, and let $\Theta:F_1\to F_2$ be a natural transformation. Let $f:X\to Y$ and $g:Y\to X$ be morphisms in $\mathcal{E}$ such that $f\circ g=\Id_Y$. Then if $\Theta_X:F_1(X)\to F_2(X)$ is an isomorphism, so is $\Theta_Y:F_1(Y)\to F_2(Y)$.
\end{lema}


\section{Definitions and first properties}


\begin{defi}  An $A$-module is called {\em semisimple} if it is the sum of its simple $A$-submodules.
A Green biset functor $A$ is called semisimple  if all $A$-modules are semisimple.
\end{defi}
\begin{lema}\label{lemma-semisimple}
A Green biset functor $A$ is semisimple if and only if for every group $H$, the $A$-module $A_H$ is  semisimple. 
\end{lema}
\begin{proof} If $A$ is semisimple, then in particular every $A$-module of the form $A_H$ is semisimple. Conversely, suppose that every $A_H$ is semisimple. As these modules can be view as the representable functors of the category $\mathcal{P}_A$, any $A$-module is a quotient of a direct sum of functors of the form $A_{H}$. Hence any $A$-module is semisimple.
\end{proof}
\begin{defi}
\begin{itemize}
\item[i)] A Green biset functor is called {\em simple} if its only two-sided ideals are itself and the zero ideal.
\item[ii)] A simple commutative Green biset functor is called a {\em Green field}.
\end{itemize}
\end{defi}

Observe that a simple Green biset functor need not be semisimple. This can be viewed by considering the case where $\mathcal{D}$ consists only of trivial groups. Then a Green biset functor over $\mathcal{D}$ is just a ring, and there are well-known examples of simple rings which are not semisimple (see e.g.~\cite{lam}, page 43, Example before Theorem 3.15).


\begin{ejem}\label{examples of Green fields}
The following Green biset functors are Green fields.
\begin{itemize}
\item If $k$ is a field of characteristic 0, the functor $kR_\rac$. Clearly it is a commutative Green biset functor and it is simple by Proposition 4.4.8 in \cite{bGfun}.
\item If $k$ is a field, the functor $kR_\mathbb{C}$.  It is clearly commutative,  and it is easy to see that the proof of Proposition 4.3 in \cite{lachica}, regarding the simplicity, holds for this functor.
\item Let $k$ be a field of characteristic 0 and $p$ be a prime number. For a finite $p$-group~$K$, let $kB^{(p)}_K$ denote the Burnside functor shifted by $K$, defined on the class $\CD$ of finite $p$-groups. It follows from~Theorem~45 of~\cite{centros} that a decomposition $1=\sum_ie_i$ of the idendity element of $kB^{(p)}_K(\un)\cong kB(K)$ as a sum of orthogonal idempotents yields a corresponding decomposition of the category $kB^{(p)}_K\Mod$ as a product of categories of modules over smaller Green biset functors $e_ikB^{(p)}_K$. Now, the primitive idempotents $e_H^K$ of $kB(K)$ are indexed by the subgroups $H$ of $K$, up to conjugation. In particular, for $H=K$, we get a Green biset functor $e_K^KkB_K^{(p)}$. It is shown in~Section~7 of~\cite{relative-B-groups} that when $K$ is not cyclic, the functor $e_K^KkB_K^{(p)}$ is a Green field.
\end{itemize}
\end{ejem}

\begin{rem}
Let $A$ be a commutative  Green biset functor. Since in this case, two-sided ideals coincide with left ideals, i.e. $A$-submodules of $A$,  then $A$ is a Green field if and only if $A$ is a simple $A$-module. 
Nevertheless, $A$ may not be, up to isomorphism, the unique simple $A$-module. For example, the functor $A=kR_{\rac}$ for a field $k$ of characteristic $0$, is a Green field, but there are infinitely many non isomorphic simple $A$-modules (see \cite{barker} or \cite{lachica}).
\end{rem}

\begin{lema}
\label{invertible}
Let $A$ be a commutative Green biset functor. Then $A$ is a Green field if and only if for any finite group $H\in\CD$, and any $a\neq 0$ in $A(H)=A(H\times \un)$, there exists $b\in A(H)=A(\un\times H)$ such that $b\circ a=\varepsilon_A$.
\end{lema}
\begin{proof}
Suppose that $A$ is a Green field. For a non zero element $a$ of $A(H)$, consider the $A$-submodule $\langle a\rangle$ of $A$ generated by $a$. Its evaluation at $L\in \CD$ is equal to $A(L\times H)\circ a$, i.e. the set of elements $b\circ a$, for $b\in A(L\times H)$. Since $\langle a\rangle$ is non zero, we have $\langle a\rangle =A$. In particular $\langle a\rangle(\un)=A(\un)$, so there exists $b\in A(\un\times H)$ such that $b\circ a=\varepsilon_A$, as was to be shown. 

For the converse, let $I$ be a non zero ideal of $A$. Let $H$ be a group such that $I(H)\neq 0$ and take $a\neq 0$ in $I(H)$. Since there exists $b\in A(H)=A(\un\times H)$ such that $b\circ a=\varepsilon_A$, then $\varepsilon_A\in I(1)$. So $I$ contains the $A$-submodule of $A$ generated by $\varepsilon_A$, which is the whole of $A$. Hence $I=A$.  
\end{proof}

If $A$ is a commutative Green biset functor, then $A(\un)$ is a commutative ring, and in fact $A$ can be viewed as a Green biset functor over $A(\un)$. This is because, more generally, if $M$ is an $A$-module and if $\alpha\in A(H\times L)$ for some groups $H,L\in \CD$, then the map $M(\alpha):M(L)\to M(H)$ is $A(\un)$-linear. This  follows from the fact that, by Proposition~39 of~\cite{centros}, the functor $A$ maps into the Green biset functor $ZA$ (the {\em center} of $A$) when $A$ is commutative, and that the evaluation $ZA(\un)$ is the endomorphism algebra of the identity functor of the category $A\Mod$.

In the same vein, we have the following lemma:
 
\begin{lema} 
\label{isfield}
Let $A$ be a commutative Green biset functor.
\begin{enumerate}
\item If $A$ is a Green field, then $A(\un)$ is a field.
\item Conversely if $A(\un)$ is a field, then $A$ has a unique maximal proper ideal. If in addition $A$ is a semisimple $A$-module, then $A$ is a Green field.
\end{enumerate}
\end{lema}
\begin{proof} Assertion 1 is a straightforward consequence of Lemma~\ref{invertible} and Remark~\ref{trivial group}.\par
Conversely, suppose that $A$ is a commutative Green biset functor such that $A(\un)$ is a field, and let $I$ be an ideal of $A$. Then $I(\un)$ is an ideal of $A(\un)$, hence $I(\un)=A(\un)$ or $I(\un)=\zero$. If $I(\un)=A(\un)$, then $I(\un)$ contains the identity element $\varepsilon_A$ of $A$. Hence $I=A$ in this case. Assume now that $I(\un)=\zero$. Then for any $H\in\CD$, the evaluation $I(H)$ is contained in the $k$-submodule $J(H)$ of $A(H)$ defined by
$$J(H)=\{u\in  A(H)=A(H\times\un)\mid\forall \alpha\in A(\un\times H),\;\alpha\circ u=0\}.$$
It is easy to see that the assignment $H\mapsto J(H)$ defines an $A$-submodule of $A$, and that $J(\un)=\zero$. It follows that $I\subseteq J$, and that $J$ is the unique maximal (proper) ideal of $A$. \par
If in addition $A$ is a semisimple $A$-module, there is an ideal $J'$ of $A$ such that $A=J\oplus J'$. If $J'\neq A$, then $J'\subseteq J$, and then $J=J+J'=A$, a contradiction. Hence $J'=A$, so $J=J\cap J'=0$. It follows that $A$ has no non zero proper ideals, hence it is a Green field.
\end{proof}

\begin{nota} Let $A$ be a Green biset functor, $M$ be an $A$-module and $G$, $H$ and $K$ be groups in $\CD$. For $a\in A(K\times G)$ and $m\in M(G\times H)$, we denote by $a\circ m$ the element of $M(K\times H)$ defined by
$$a\circ m=M\big(\Def_{K\times H}^{K\times \Delta(G)\times H}\big)M\big(\Res_{K\times \Delta(G)\times H}^{K\times G\times G\times H}\big)(a\times m).$$
\end{nota}
Using Definition 10 in \cite{centros}, one can see that $a\circ m=M(a\times H)(m)$, but we will not use this fact. In the case $M=A$, this notation is consistent with the definition of the composition in the category $\CP_A$ given in Definition~\ref{PA}. Moreover:
\begin{lema}\label{assoc} Let $A$ be a Green biset functor and $M$ be an $A$-module. Let moreover $G$, $H$ and $K$ be finite groups in $\CD$. Then for any $a\in A(K\times G)$, $\alpha\in A(G)$ and $m\in M(H)$,
$$a\circ(\alpha\times m)=(a\circ\alpha)\times m$$
in $M(K\times H)$.
\end{lema}
\begin{proof} Indeed
\begin{align*}
a\circ(\alpha\times  m)&=M(\Def_{K\times H}^{K\times\Delta(G)\times H})M(\Res_{K\times\Delta(G)\times H}^{K\times G\times G\times H})(a\times \alpha\times m)\\
&=M(\Def_{K\times H}^{K\times\Delta(G)\times H})\big(A(\Res_{K\times\Delta(G)}^{K\times G\times G})(a\times \alpha)\times m\big)\\
&=\big(A(\Def_{K}^{K\times\Delta(G)})A(\Res_{K\times\Delta(G)}^{K\times G\times G})(a\times \alpha)\big)\times m\\
&=(a\circ\alpha)\times m,
\end{align*}
as was to be shown. 
\end{proof}
\begin{prop} Let $A$ be a Green field and $M$ be an $A$-module. Then for any finite groups $G$ and $H$ in $\CD$, the linear map
$$\pi_{G,H}:A(G)\otimes_{A(\un)}M(H)\to M(G\times H)$$
sending $\alpha\otimes m$ to $\alpha\times m$ is injective.
\end{prop}
\begin{proof} Let $u$ be a non zero element of the kernel of $\pi_{G,H}$, and $n\geq 1$ be the smallest integer such that $u$ can be written $u=\sum_{i=1}^n\limits\alpha_i\otimes m_i$, for elements $\alpha_i\in A(G)$ and $ m_i\in M(H)$. Then all the elements $\alpha_i$ and $ m_i$ are non zero. In particular, by Lemma~\ref{invertible}, there exists $a\in A(G)=A(\un\times G)$ such that $a\circ\alpha_n=\varepsilon_A$. By Lemma~\ref{assoc}, we get
\begin{align*}0&=a\circ \sum_{i=1}^n(\alpha_i\times m_i)=\sum_{i=1}^na\circ(\alpha_i\times m_i)\\
&=\sum_{i=1}^n(a\circ\alpha_i)\times m_i=\big(\sum_{i=1}^{n-1}(a\circ\alpha_i)\times m_i\big) +(\varepsilon_A\times m_n).
\end{align*}
It follows that $ m_n=-\sum_{i=1}^{n-1}\limits(a\circ\alpha_i)\times m_i$, so 
\begin{align*}u&=\sum_{i=1}^{n-1}\alpha_i\otimes m_i-\alpha_n\otimes \sum_{i=1}^{n-1}\big((a\circ\alpha_i)\times m_i\big)\\
&=\sum_{i=1}^{n-1}\alpha_i\otimes m_i-\sum_{i=1}^{n-1}\alpha_n\otimes\big((a\circ\alpha_i)\times m_i\big)\\
&=\sum_{i=1}^{n-1}\alpha_i\otimes m_i-\sum_{i=1}^{n-1}\big(\alpha_n\times(a\circ\alpha_i)\big)\otimes m_i,\textrm{ since }a\circ\alpha_i\in A(\un),\\
&=\sum_{i=1}^{n-1}\Big(\alpha_i-\big(\alpha_n\times(a\circ\alpha_i)\big)\Big)\otimes m_i.
\end{align*}
This contradicts the minimality of $n$ and completes the proof.
\end{proof}

A relevant question is, of course, to know whether a Green field is always semisimple.
 To address this question  we introduce the following notation and results. Unfortunately, as we  will see in Example \ref{trivsou}, a general  answer seems difficult to find.

\begin{nota} 
Let $A$ be a Green biset functor. For a finite group $L$, we denote by $t_L$ the linear map $A(\Def_\un^L):A(L)\to A(\un)$, and by $\langle-,-\rangle_L$ the bilinear map
$$(u,v)\in A(L)\times A(L)\mapsto t_L(u\cdot v)\in A(\un).$$ 
When considering the bilinear map $\langle-,-\rangle_{H\times L}$, defined in $A(H\times L)$, for finite groups $H$ and $L$, we will write  $\langle-,-\rangle_{H,L}$ instead of $\langle-,-\rangle_{H\times L}$. 
\end{nota}

\begin{rem}
If $A$ is a commutative Green biset functor, the bilinear map $\langle-,-\rangle_L$ is symmetric. This is clear since the ring $\big(A(L),\cdot\big)$ is commutative.
\end{rem}

\begin{lema} \label{bilinear maps}Let $A$ be a Green biset functor and let $L$ and $H$ be  groups in $\CD$. Then, for any $\alpha,\beta\in A(H\times L)$
$$\langle\alpha,\beta\rangle_{H, L}=A(\Def_\un^{\Delta (L)}\Res_{\Delta (L)}^{L\times L})(\alpha\op\circ\beta).$$
 \end{lema}
\begin{proof} 
In the following computations, we will omit the $\times$ signs in the direct products of groups, thus writing e.g. $HLL$ instead of $H\times L\times L$.

Consider first $A(\Def_\un^{\Delta(L)}\Res_{\Delta(L)}^{LL})(\alpha\op\circ\beta)$. This is equal to
\begin{displaymath}
A(\Def_\un^{\Delta(L)}\Res_{\Delta(L)}^{LL}\Def_{LL}^{L\Delta(H)L}\Res_{L\Delta(H)L}^{LHHL})(\alpha\op\times\beta).
\end{displaymath}
Now $\Res_{\Delta(L)}^{LL}\Def_{LL}^{L\Delta(H)L}\cong \Def^{D}_{\Delta(L)} \Res^{L\Delta(H)L}_{D}$, with $D=\{(l,\,h,\,h,\,l)\mid h\in H,\ l\in L\}$, and $\alpha\op\times\beta=A(\Iso(\theta))(\alpha\times\beta)$, where $\theta:HLHL\to LHHL$ swaps the first two components. Hence
\begin{align*}
A(\Def_\un^{\Delta(L)}\Res_{\Delta(L)}^{LL})(\alpha\op\circ\beta)&=A(\Def_\un^{\Delta(L)}\Def^{D}_{\Delta(L)} \Res^{L\Delta(H)L}_{D}\Res_{L\Delta(H)L}^{LHHL}\Iso(\theta))(\alpha\times\beta)\\
&=A(\Def_\un^{HL}\Iso^D_{HL}\Res_{D}^{LHHL}\Iso(\theta))(\alpha\times\beta)\\
&=A(\Def_\un^{HL}\Iso^{\Delta(HL)}_{HL}\Res_{\Delta(HL)}^{HLHL})(\alpha\times\beta)\\
&=A(\Def_\un^{HL})(\alpha\cdot\beta)=t_{HL}(\alpha\cdot\beta)\\
&=\langle \alpha,\beta\rangle_{H,L}
\end{align*}
as desired.
\end{proof}


The bilinear map just defined allows us to prove the following characterization of a Green field.

\begin{prop} \label{non degenerate} Let $A$ be a commutative Green biset functor. The following conditions are equivalent:
\begin{enumerate}
\item $A$ is a Green field.
\item $A(\un)$ is a field and the bilinear map $\langle-,-\rangle_{H,\un}$ is non degenerate, for any $H\in\CD$.
\item $A(\un)$ is a field and the bilinear map $\langle-,-\rangle_{H,L}$ is non degenerate, for any $H,L\in\nolinebreak\CD$.
\end{enumerate}
\end{prop}
\begin{proof}
By Lemma \ref{isfield} we know that if $A$ is a Green field, then $A(1)$ is a field.\mpn
$1\Rightarrow 2$: Let $a,b\in A(H\times \un)$, we have $\langle b,a\rangle_{H,\un}=\Def_\un^{\un}\Res_{\un}^{\un\times\un}(b\op\circ a)=b\op\circ a$. So if 1) holds, by Lemma \ref{invertible}, the radical of $\langle-,-\rangle_{H,\un}$ has to be zero, and 2) holds.\mpn
$2\Leftrightarrow 3$: Suppose 2) holds. By definition, for $\alpha,\beta\in A(H\times L)$, we have 
$$\langle \alpha,\beta\rangle_{H,L}=t_{H\times L}(\alpha\cdot\beta)=\langle \alpha,\beta\rangle_{H\times L,\un},$$
and the bilinear map $\langle \alpha,\beta\rangle_{H\times L,\un}$ is non degenerate. The converse  is clear.\mpn
$2\Rightarrow 1$: Suppose that 2) holds and let $I$ be an ideal of $A$. Then $I(\un)$ is an ideal of the field $A(\un)$, thus $I(\un)=A(\un)$ or $I(\un)=\zero$. If $I(\un)=A(\un)$, then $I(\un)\ni\varepsilon_A$, so $I=A$ in this case. Now if $I(\un)=\zero$, then for any $H\in\CD$, we have that $A(\un\times H)\circ I(H)\subseteq I(\un)=\zero$, so $A(\un\times H)\circ I(H)=\zero$. It follows that $\langle A(H),I(H)\rangle_{H,\un}=\zero$, so $I(H)=\zero$ since $\langle-,-\rangle_{H,\un}$ is non degenerate. Hence $I=0$ in this case, and $A$ is a Green field.
\end{proof}
\begin{rem} In the case $A(H)$ is a finite dimensional $A(\un)$-algebra, the condition 2 above means that $t_H$ is a {\em symmetrizing form} for this algebra, i.e. that the bilinear form $\langle-,-\rangle_H$ yields an isomorphism of $\big(A(\un),A(\un)\big)$-bimodules between $A(H)$ and $\Hom_{A(\un)}\big(A(H),A(\un)\big)$. In particular $A(H)$ is a symmetric algebra in this case (see~\cite{broue_higman} for details).
\end{rem}


We have one last interesting property of the map $\langle-,-\rangle_{H,L}$.

\begin{lema}\label{associative}
Let $A$ be a commutative Green biset functor, and $H,K,L\in \CD$. Then for any $\alpha\in A(H\times K)$, $\beta\in A(L\times K)$ and $\gamma \in A(L\times H)$
$$\langle \gamma\circ\alpha,\beta\rangle_{L,K}=\langle\alpha, \gamma\op\circ\beta\rangle_{H,K}.$$
In particular, the assignment $H\mapsto\Rad\,\langle-,-\rangle_H$ defines an ideal of $A$.
\end{lema}
\begin{proof}
By Lemma~\ref{bilinear maps} and Lemma~\ref{opposite}, we have
\begin{align*}
\langle \gamma\circ\alpha,\beta\rangle_{L,K}&=t_{KK}\big((\gamma\circ\alpha)\op\circ\beta\big)=t_{KK}\big((\alpha\op\circ\gamma\op)\circ\beta)\big)\\
&=t_{KK}\big(\alpha\op\circ(\gamma\op\circ\beta)\big)=\langle \alpha,\gamma\op\circ\beta\rangle_{H,K},
\end{align*}
as was to be shown.\mpn
In particular for $K=\un$,  if $\alpha\in \Rad\langle-,-\rangle_H$ and $\gamma\in A(L\times H)$, then $\gamma\circ \alpha\in \Rad\langle -,-\rangle_L$. 
\end{proof}


The following result is Proposition~7 in \cite{shifted}, we include the proof here for the convenience of the reader.  We  obtain, as a corollary, a necessary and sufficient condition for a Green field to be semisimple.


\begin{prop}\label{implies semisimple}
Let $A$ be a Green biset functor over $k$. Assume the following:
\begin{enumerate}
\item If $L$ is a finite group  in $\CD$, the algebra  $A(L\times L)$ is semisimple.
\item If $H$ is a finite group  in $\CD$, the bilinear map
$$(\alpha,\beta)\in A(H\times L)^2\mapsto \alpha^{\rm op}\circ\beta\in A(L\times L)$$
is non degenerate, i.e. if $\alpha\op\circ\beta=0$ for all $\alpha\in A(H\times L)$, then $\beta=0$.
\end{enumerate}
Then $A$ is a semisimple Green biset functor. 
\end{prop}
\begin{proof}
We will show that $A_L$ is a semisimple $A$-module, for each finite group $L$. Let $M$ be an $A$-submodule of $A_L$. 
\begin{itemize}
\item First, $M(L)$ is a left ideal of the algebra $A(L\times L)$, and we claim that if $M(L)=\zero$, then $M=0$. Indeed, if $M(L)=\zero$, then for any finite group $H$, the $k$-vector space $M(H)$ is contained in the set of elements $\beta$ of $A_L(H)=A(H\times L)$ such that $\alpha\op\circ\beta=0$ for each $\alpha\in A(H\times L)$ (since $\alpha\op\circ \beta\in M(L)=\zero$). By assumption 2, we have $M(H)=\zero$, as claimed. 
\item Now let $\Phi$ be the correspondence sending an $A$-submodule $M$ of $A_L$ to the (left) ideal $M(L)$ of the algebra $A(L\times L)$. In the other direction, let $\Psi$ be the correspondence sending a left ideal $V$ of $A(L\times L)$ to the $A$-submodule $\langle V\rangle$ of $A_L$ generated by $V$.\par
Clearly $\Phi\circ\Psi(V)=V$, since $\langle V\rangle(L)=A(L\times L)(V)=V$.\par
  Conversely, let $M$ be an $A$-submodule of $A_L$. Then $M(L)$ is a left ideal of $A(L\times L)$, which is semisimple by assumption 1. Then there exists a left ideal  $W$ of $A(L\times L)$ such that
$$M(L)\oplus W=A(L\times L).$$
Let $P=M+\Psi(W)$. The evaluation of $P$ at $L$ is 
$$P(L)=M(L)+W=A(L\times L).$$
Then $P=A_L$, because $A_L$ is generated by $A_L(L)=A(L\times L)$.  Moreover, the intersection $I=M\cap\Psi(W)$, evaluated at $L$, is equal to
$$I(L)=M(L)\cap W=\zero.$$
It follows that $I=\zero$, and then $M\oplus \Psi(W)=A_L$.  \par
Now consider $M'=\Psi\circ\Phi(M)$. This is an $A$-submodule of $M$, and the same arguments show that $M'\oplus \Psi(W)=A_L$ also. But we have
$$M'\leq M\leq M'\oplus \Psi(W),$$
and then $M=M'\oplus\big(M\cap \Psi(W)\big)=M'\oplus I=M'$.\par
We have shown that $\Phi$ and $\Psi$ are mutually inverse bijections between the poset of $A$-submodules of $A_L$ and the poset of left ideals of $A(L\times L)$. It follows that $A_L$ is a semisimple $A$-module, for each finite group $L$. By Lemma~\ref{lemma-semisimple}, $A$ is semisimple.
\end{itemize}
\end{proof}


\begin{coro} \label{ALL semisimple}Let $A$ be a Green field. Then $A$ is semisimple if and only if the algebra $A(L\times L)$ is semisimple for any $L\in\CD$.
\end{coro}
\begin{proof} Indeed if $A$ is semisimple, then $A_L$ is a semisimple $A$-module for any $L\in \CD$, so $\End_{A\Mod}(L)\cong A(L\times L)$ is a semisimple algebra. Conversely, by Lemma \ref{bilinear maps} and Proposition \ref{non degenerate}, a Green field always fulfills condition 2 of Proposition~\ref{implies semisimple}. Hence for a Green field, condition 1 alone implies semisimplicity.
\end{proof}

\begin{ejem} 
\label{trivsou}
Let $k$ be a field of characteristic $p>0$. The Green ring functor $a_k$ is the Green biset functor (over $\Z$) sending a finite group $G$ to the Grothendieck group of the category of finitely generated $kG$-modules, for relations given by direct sum decompositions. For finite groups $G$ and $H$, the product $\times:A(G)\times A(H)\to A(G\times H)$ is induced by the usual external tensor product: if $M$ is a $kG$-module and $N$ is a $kH$-module, then $M\otimes_kN$ has a natural structure of $k(G\times H)$-module.\par
For three groups $G$, $H$, and $K$, the composition $a_k(K\times H)\times a_k(H\times G)\to a_k(K\times G)$ is induced by the tensor product of bimodules: a $k(K\times H)$-module $N$ can be viewed as a $(kK,kH)$-bimodule, and a $k(H\times G)$-module $M$ can be viewed as a $(kH,kG)$-bimodule. Then the composition of (the class of) $N$ and (the class of) $M$ in $a_k(K\times G)$ is (the class of) $N\otimes_{kH}M$.\par
 Now let $\F$ be a field of characteristic 0 and $\F a_k=\F\otimes a_k$ be the Green biset functor obtained by extending the coefficients to $\F$. Then $\F a_k(\un)$ is the field $\F$. Moreover, it follows from Corollary 5.11.2 of \cite{benson} that the bilinear form $\langle-,-\rangle_{H,\un}$ is non degenerate, for any finite group $H$. By Proposition~\ref{non degenerate}, the functor $\F a_k$ is a Green field. For short, we could say that {\em the Green ring is a Green field}.\par
This example shows that the question of semisimplicity of Green fields is not that simple: we don't know if the algebra $\big(\F a_k(G\times G),\circ\big)$ is always semisimple, and we leave this question open. We recall that, on the other hand, the commutative algebra $\big(a_k(G),\cdot\big)$ is non reduced in general - that is, its nilradical is non zero (see Section 5.8 of~\cite{benson} for details).
\end{ejem}

We finish the section with an application of the previous corollary.

\begin{coro} \label{anisotropic}Let $A$ be a Green field. Assume that $\dim_{A(\un)}A(L\times L)<+\infty$ and that the bilinear form $\langle\,,\,\rangle_{L,L}$ is anisotropic, for any $L$ in $\CD$. Then the functor $A$ is semisimple.
\end{coro}
\begin{proof} By Corollary~\ref{ALL semisimple}, it suffices to show that the algebra $A(L\times L)$ is semisimple, for any $L\in \CD$. Let $L\in \CD$ and $I$ be a left ideal of $A(L\times L)$. Let $J$ be the left orthogonal of $I$ in $A(L\times L)$, that is
$$J=\{u\in A(L\times L)\mid \forall v\in I,\;\langle u,v\rangle_{L,L}=0\}.$$
Then $J$ is a left ideal of $A(L\times L)$, by Lemma~\ref{associative}. Moreover $I\cap J=\{0\}$ since the form $\langle\,,\,\rangle_{L,L}$ is anisotropic. In addition $\dim_{A(\un)}I+\dim_{A(\un)}J=\dim_{A(\un)}A(L\times L)$ since the form $\langle\,,\,\rangle_{L,L}$ is non-degenerate. It follows that $I\oplus J=A(L\times L)$, so the algebra $A(L\times L)$ is semisimple, as was to be shown.
\end{proof}

\begin{ejem} Let $A=\Q R_\Q$. Then for any $L\in\CD$ and any finite dimensional $\Q(L\times L)$-modules $M$ and $N$, one checks easily that $\langle M,N\rangle_{L,L}$ is equal to the dimension of the space of co-invariants $(M\otimes_\Q N)_{L\times L}$ of $L\times L$ on the tensor product $M\otimes_\Q N$. This is also equal to the dimension of the space  $(M\otimes_\Q N)^{L\times L}$, and thus can be computed using the characters $\chi_M$ and $\chi_N$ of $M$ and $N$, respectively, as
$$\langle M,N\rangle_{L,L}=\frac{1}{|L|^2}\sum_{g,h\in L}\chi_M(g,h)\chi_N(g,h).$$
Now the map $M\mapsto \chi_M$ from $R_\Q(L\times L)$ to the space $cf_{\Q}(L\times L)$ of class functions $L\times L\to \Q$ extends to a linear injective map $u\mapsto \chi_u$ from $\Q R_\Q(L\times L)$ to $cf_{\Q}(L\times L)$, and for any $u,v\in \Q R_\Q(L\times L)$, we get that
$$\langle u,v\rangle_{L,L}=\frac{1}{|L|^2}\sum_{g,h\in L}\chi_u(g,h)\chi_v(g,h).$$
In particular $\langle u,u\rangle_{L,L}\geq 0$, with equality if and only if $\chi_u=0$, that is $u=0$. In other words the bilinear form $\langle \,,\,\rangle_{L,L}$ is anisotropic, and Corollary~\ref{anisotropic} can be applied. This gives another proof of the fact that the Green biset functor $\Q R_\Q$ is semisimple - this was first proved by L. Barker in~\cite{barker}. Variations on this argument were used in~\cite{shifted}, to show that all the shifted functors $(kR_\F)_T$ are semisimple, when $k$ and $\F$ are fields of characteristic 0, and $T$ is any finite group.
\end{ejem}


\section{Strict Green fields}
\label{estrictos}

Recall that we have fixed a  non-empty class $\mathcal{D}$ of finite groups, closed  under subquotients and direct products. In particular, the trivial group belongs to $\mathcal{D}$. We denote by $A$ a Green biset functor defined on $\mathcal{D}$, over a commutative unital ring $k$. 

\begin{teo}
\label{equiv1}
The following conditions are equivalent:
\begin{enumerate}
\item If $M$ is a non-zero $A$-module, then $M(\un)\neq \zero$.
\item If $S$ is a simple $A$-module, then $S(\un)\neq\zero$.
\item If $S$ is a simple $A$-module, then $S\cong S_{\un,V}$, where $V$ is a simple $A(\un)$-module.
\item For any finite group $H$ in $\mathcal{D}$, there exists a positive integer $n_H$ such that $A_H$ is a direct summand of $A^{\oplus n_H}$.
\item The evaluation functor, which sends an $A$-module $M$ to $M(\un)$, is an equivalence of categories between  $\gMod{A}$ and $\gMod{A(\un)}$.
\item For any finite groups $G$ and $H$ in $\mathcal{D}$, the product $\times:A(G)\times A(H)\to A(G\times H)$
induces an isomorphism of $k$-modules between  $A(G)\otimes_{A(\un)}A(H)$ and $A(G\times H)$.
\end{enumerate}
\end{teo}
\begin{proof}
$1\Rightarrow 2{:}$ This is trivial, since a simple $A$-module is non-zero by definition.\mpn
$2\Rightarrow 3{:}$ If $S$ is a simple $A$-module and $H$ is a finite group such that $S(H)\neq\zero$, then $V=S(H)$ is a simple $A(H\times H)$-module, and $S\cong S_{H,V}$. If 2 holds, we can take $H=\un$, and this gives 3.\mpn
$3\Rightarrow 4{:}$ Let $H$ be a finite group such that $\widehat{A}(H)\neq \zero$. Then we can take a simple $\widehat{A}(H)$-module $V$ and construct the simple $A$-module $S_{H,\, V}$, having $H$ as a minimal group. If $H$ is non-trivial, it satisfies $S_{H,\, V}(\un)=0$. So if 3 holds, we have $\widehat{A}(H)=\zero$ for any non-trivial finite group $H$. 
Then, any element in  $A(H\times H)$ can be written as $\sum_{i=1}^na_i\circ b_i$, where $a_i\in A(H\times K_i)$ and $b_i\in A(K_i\times H)$, for some $K_i$ with $|K_i|<|H|$. In particular, using the Yoneda Lemma, the identity $\Id_{A_H}$ in $\mathrm{End}_A(A_H)$ can be written as $\sum_{i=1}^n\beta_i\circ\alpha_i$, where $\alpha_i:A_{H}\rightarrow A_{K_i}$ and $\beta_i:A_{K_i}\rightarrow A_H$, for some $K_i$ with $|K_i|<|H|$. Now the elements $\alpha_i$ define a morphism $\alpha:A_H\to\mathop{\oplus}_{i=1}^{n_H}\limits A_{K_i}$, and the elements $\beta_i$ define a morphism $\mathop{\oplus}_{i=1}^{n_H}\limits A_{K_i}\to A_H$. Saying that $\Id=\sum_i\alpha_i\circ\beta_i$ is equivalent to saying that $\beta\circ\alpha=\Id_{A_H}$, so $A_H$ is a direct summand of $\mathop{\oplus}_{i=1}^{n_H}\limits A_{K_i}$.  Applying this argument now to each $A_{K_i}$ and then proceeding by induction on the order of $H$, it follows that $A_H$ is a direct summand of a finite direct sum of copies of $A_\un\cong A$. Hence 3 implies 4.\mpn
$4\Rightarrow 5{:}$ For an $A$-module $M$, the counit of the adjunction $L_{\un,-}\dashv ev_\un$, at the group $G$ is the map $A(G)\otimes_{A(\un)} M(\un)\to M(G)$ induced by the product $\times$ on $M$. If $M=A$, this map is an isomorphism. Hence  it is an isomorphism if $M$ is a direct summand of a direct sum of a finite number of copies of $A$, by Lemma \ref{section}. In particular, if 4 holds, this map is an isomorphism for any $M=A_H$.  Hence, it is an isomorphism if $M$ is a projective $A$-module. \par
If $M$ is an arbitrary $A$-module, there is a resolution $Q\to P\to M\to 0$ by projective $A$-modules $P$  and~$Q$, and as the functor $M\mapsto L_{\un,M(\un)}$ is right exact, we get a commutative diagram
$$\xymatrix{
L_{\un,Q(\un)}\ar[d]^\cong\ar[r]&L_{\un,P(\un)}\ar[d]^\cong\ar[r]&L_{\un,M(\un)}\ar[d]\ar[r]&0\\
Q\ar[r]&P\ar[r]&M\ar[r]&0
}$$
with exact rows. The two left vertical arrows are isomorphims, so the right vertical arrow is an isomorphism. Hence the counit of the adjunction $L_{\un,-}\dashv ev_\un$ is an isomorphism. Since $ev_\un(L_{\un,V})=L_{\un,V}(\un)\cong V$ for any $A(\un)$-module $V$, the unit of the adjunction $L_{\un,-}\dashv ev_\un$ is also an isomorphism. So the functors $ev_\un$ and $L_{\un,-}$ are quasi-inverse equivalences of categories, and 5 holds.\mpn 
$5\Rightarrow 6{:}$ Let $M=A_H$. If 5 holds, the counit $L_{\un,M(\un)}\to M$ is an isomorphism. At the trivial group, this is precisely the map $A(G)\otimes_{A(\un)}A(H)\to A(G\times H)$ of assertion 6.\mpn
$6\Rightarrow 5{:}$ If 6 holds, then the counit $L_{\un,M(\un)}\to M$ is an isomorphism for any $A$-module of the form $A_H$. Hence it is an isomorphism for any projective $A$-module, hence for any $A$-module, as above. It follows that $ev_\un$ and $L_{\un,-}$ are quasi-inverse equivalences of categories, so 5 holds.\mpn 
$5\Rightarrow 1{:}$ If 5 holds, then $M\cong L_{\un, M(\un)}$ for any $A$-module $M$. Hence $M(\un)\neq \zero$ if $M\neq \zero$.
\end{proof}

\begin{coro} \label{A simple}Let $A$ be a Green biset functor, and suppose that $A$ is a simple $A$-module. The following are equivalent:
\begin{enumerate}
\item The only simple $A$-module, up to isomorphism, is $A$.
\item For any finite group $H$, there exists a positive integer $n_H$ such that $A_H\cong A^{\oplus n_H}$.
\item $A(\un)$ is a division ring, and the evaluation functor $ev_{\un}$ defines an equivalence of categories between $\gMod{A}$ and $A(\un)\Vect$. 
\end{enumerate}
\end{coro}
\begin{proof} $1\Rightarrow 2{:}$ If 1 holds, then $S(\un)\neq\zero$ for any simple $A$-module $S$. So, condition 2 of Theorem~\ref{equiv1} is fulfilled. In particular, the functor $A_H$ is a direct summand of some finite direct sum $A^{\oplus m_H}$. Since $A$ is simple, it follows that $A_H$ is isomorphic to a finite direct sum of copies of $A$.\mpn
$2\Rightarrow 3{:}$ If 2 holds, then condition 5 of Theorem~\ref{equiv1} is fulfilled. In particular, the evaluation functor $ev_\un$ is an equivalence of categories. The image of the simple $A$-module $A$ by this equivalence is the $A(\un)$-module $A(\un)$, so $A(\un)$ is a simple $A(\un)$-module. If follows that $A(\un)$ is a division ring, and 3 holds.\mpn
$3\Rightarrow 1{:}$ If 3 holds, then condition 6 of Theorem~\ref{equiv1} is fulfilled. In particular any simple $A$-module is of the form $S_{\un,V}$, where $V$ is a simple $A(\un)$-module. Since $A(\un)$ is a division ring, the only simple $A(\un)$-module is $V=A(\un)$. Hence $S_{\un,A(\un)}$ is the only simple $A$-module, isomorphic to $A$ since $A$ is a simple $A$-module.
\end{proof}


\begin{defi}
Green fields fulfilling any of the conditions of Corollary \ref{A simple} will be called \textit{strict} Green fields.
\end{defi}

\begin{rem} 
It follows from Corollary~\ref{A simple} that a strict Green field is semisimple.
\end{rem}

\subsection*{Examples}

1) Our first example of Theorem~\ref{equiv1} is the case where $A$ is the functor $(R_\C)_K$ of complex characters, shifted by a finite group $K$.  It is easy to see that Proposition 4.2  in Garc\'ia \cite{garcia-shifted} holds for $A=(R_\C)_K$, showing that $A$ satisfies condition 6 of Theorem~\ref{equiv1}. Since $A(\un)$ is isomorphic to the ring $R_\C(K)$, it follows that the evaluation functor $ev_\un:M\mapsto M(\un)$ is an equivalence of categories between $\gMod{A}$ and $\gMod{R_\C(K)}$.  

More generally, for any commutative ring $k$, we get an equivalence of categories between $\gMod{k(R_\C)_K}$ and $\gMod{kR_\C(K)}$. In particular, when $k$ is a field whose characteristic does not divide the order of $K$, the category $\gMod{k(R_\C)_K}$ is semisimple.  If in addition to this we take $K=\un$, then $kR_\C(\un)=k$. Hence $kR_\C$ satisfies condition 3 of Corollary \ref{A simple} and it is a strict Green field. 

It is worth saying that the fact that $\C R_\C$ satisfies condition 1 of Corollary \ref{A simple} was first shown in Proposition 4.3 of \cite{lachica}. Also, the equivalence by evaluation at $\un$ between $\gMod{\C (R_\C)_K}$ and $\gMod{\C R_\C(K)}$ was first given in Proposition 4.3 of \cite{garcia-shifted}.

\begin{rem} The ring $A(\un)$ may indeed be a division ring, for example by taking $k$ as the quaternion algebra over $\mathbb{Q}$ and $A=kR_{\C}$. As we have defined Green biset functors only over commutative rings, in this case, we must consider $A$ as a Green functor over~$\mathbb{Q}$.
\end{rem}

\noindent2) Let $k$ be a field of prime characteristic $q$. Assume that the class $\mathcal{D}$ consists of finite groups whose orders are products of primes all congruent to 1 mod $q$. Then there  exists a biset functor $\sou{k}$, defined on $\mathcal{D}$, such that $\sou{k}(G)=k$ for any $G\in\mathcal{D}$, and $\sou{k}(U):k\to k$ is multiplication by $|H\dom U|$ for any $G$, $H$ in $\mathcal{D}$ and any finite $(H,G)$-biset~$U$: indeed for any finite $G$-set $X$, the cardinality of each orbit $Gx$ on $X$ is congruent to 1 mod~$q$, so $|G\dom X|\equiv |X|\; (\mathrm{mod}\, q)$. It follows that if $K$ is a group in $\mathcal{D}$, and $V$ is a finite $(K,H)$-biset, then 
$$|K\dom(V\times_HU)|\equiv |V\times_HU|\equiv |V\times U|\equiv |K\dom V||H\dom U| \; (\mathrm{mod}\, q),$$
so $\sou{k}$ is indeed a biset functor over $k$, defined on $\mathcal{D}$. In the case where $\mathcal{D}$ consists of finite $p$-groups, for a fixed prime number $p$ congruent to 1 mod $q$, the functor $\sou{k}$ has been considered in~\cite{both}, Corollary 8.4.\par
Now it is clear that the multiplication in the field $k$ induces a Green functor structure on $\sou{k}$, and condition 6 of Theorem~\ref{equiv1} holds. Again Corollary~\ref{A simple}  also holds and $\sou{k}$ is a strict Green field.
\mpn

\noindent3) Let $p$ be a prime, and $K$ be a non-cyclic $p$-group. The Green biset functor $A=e_K^KkB_K^{(p)}$ considered in Example~\ref{examples of Green fields} is a non strict Green field, because it does not satisfy Condition 6 of Theorem~\ref{equiv1}. Indeed, for a $p$-group $G$, the dimension of $A(G)$ is equal to the number of conjugacy classes of subgroups of $G\times K$ which map surjectively to $K$ by the second projection $G\times K\to K$ (see Section~5.2.2 of~\cite{centros}). In the case $K=C_p\times C_p$ and $G=C_p$, an easy computation shows that $\dim_kA(G)=p^2+1$, whereas $\dim_kA(G\times G)=p^4+p^3+p^2+1$. Since
$$(p^4+p^3+p^2+1)-(p^2+1)^2=p^2(p-1)>0,$$
the functor $A$ does not satisfy Condition 6 of Theorem~\ref{equiv1}.\mpn



\centerline{\rule{5ex}{.1ex}}
\begin{flushleft}
Serge Bouc, CNRS-LAMFA, Universit\'e de Picardie, 33 rue St Leu, 80039, Amiens, France.\\
{\tt serge.bouc@u-picardie.fr}\vspace{1ex}\\
Nadia Romero, DEMAT, UGTO, Jalisco s/n, Mineral de Valenciana, 36240, Guanajuato, Gto., Mexico.\\
{\tt nadia.romero@ugto.mx}
\end{flushleft}

\end{document}